\documentclass[reqno,12pt]{amsart}
\usepackage{amsmath,amsthm}
\usepackage[T1]{fontenc}
\usepackage{enumerate}
\theoremstyle{plain}
\newtheorem{thm}{Theorem}

\newtheorem{lemma}{Lemma}

\newtheorem{definite}{Definition}

\newcommand{\ra}{\rightarrow}

\begin{document}
\title{On half Cauchy sequences}
\author[Frank J. Palladino]{Frank J. Palladino}
\date{January 24, 2011}

\begin{abstract}
\noindent 
In this note we introduce and define half Cauchy sequences. We prove that a sequence of real numbers is convergent if and only if it is bounded and half Cauchy. 
We also provide an example of how the concept may be used.

\end{abstract}
\maketitle

\section{Introduction}
Bounded monotone sequences of real numbers converge. This is an indispensable theorem in real analysis which is used to prove convergence in a wide variety of situations. Since this is such a useful idea, why not weaken the hypothesis of the theorem?
A sequence $\{a_n\}^{\infty}_{n=1}$ is monotone non-increasing if $a_{n+1}-a_{n}\leq 0$. However, for our purposes, $0$ is an unnecessarily strict reqirement on the right hand side of the inequality. We wish to find a property weaker than  monotonicity but strong enough to guarantee convergence under similar circumstances. 
This becomes a problem of restricting upward motion of our sequence $\{a_n\}^{\infty}_{n=1}$ by just enough to guarantee convergence.\newline
Since we would like sequences with our yet to be determined property to converge, such sequences must be Cauchy, even if it is not apparent from the outset. So upward motion is naturally restricted by the Cauchy condition. 
Therefore, our sequences must have the property that for every $\epsilon > 0$ there exists an $N\in\mathbb{N}$ so that for $n\geq m \geq N$, $a_{n}-a_{m}<\epsilon$. This reasoning provides the motivation for the forthcoming definition of a half Cauchy sequence.

\section{Basic Properties of half Cauchy sequences}
\begin{definite}
We say that a sequence $\{a_n\}^{\infty}_{n=1}$ in $\mathbb{R}$ is upward half Cauchy if for every $\epsilon > 0$ there exists an $N\in\mathbb{N}$ so that for $n\geq m \geq N$, $a_{m}-a_{n}<\epsilon$.\newline 
\end{definite}
\begin{definite}
We say that a sequence $\{a_n\}^{\infty}_{n=1}$ in $\mathbb{R}$ is downward half Cauchy if for every $\epsilon > 0$ there exists an $N\in\mathbb{N}$ so that for $n\geq m \geq N$, $a_{n}-a_{m}<\epsilon$.\newline
\end{definite}
\begin{definite}
We say that a sequence $\{a_n\}^{\infty}_{n=1}$ in $\mathbb{R}$ is half Cauchy if the sequence is either upward half Cauchy, or downward half Cauchy, or both. 
\end{definite}
\begin{thm}
A sequence $\{a_n\}^{\infty}_{n=1}$ in $\mathbb{R}$ is Cauchy if and only if it is both upward half Cauchy and downward half Cauchy.
\end{thm}
\begin{proof}
Suppose that a sequence $\{a_n\}^{\infty}_{n=1}$ in $\mathbb{R}$ is Cauchy then choose an arbitrary $\epsilon > 0$. 
Since $\{a_n\}^{\infty}_{n=1}$ is Cauchy there exists an $N\in\mathbb{N}$ so that for $n\geq m \geq N$, $|a_{m}-a_{n}|<\epsilon$. 
Using this same $N\in\mathbb{N}$ we have that for $n\geq m \geq N$, $a_{m}-a_{n}\leq |a_{m}-a_{n}|<\epsilon$ and $a_{n}-a_{m}\leq |a_{m}-a_{n}|<\epsilon$. 
Thus $\{a_n\}^{\infty}_{n=1}$ is both upward half Cauchy and downward half Cauchy. Now suppose $\{a_n\}^{\infty}_{n=1}$ is both upward half Cauchy and downward half Cauchy then choose an arbitrary $\epsilon > 0$.
Since $\{a_n\}^{\infty}_{n=1}$ is upward half Cauchy, there exists an $N_{1}\in\mathbb{N}$ so that for $n\geq m \geq N_{1}$, $a_{m}-a_{n}<\epsilon$. 
Since $\{a_n\}^{\infty}_{n=1}$ is downward half Cauchy, there exists an $N_{2}\in\mathbb{N}$ so that for $n\geq m \geq N_{2}$, $a_{n}-a_{m}<\epsilon$. 
Using $N=\max(N_{1},N_{2})\in \mathbb{N}$ we get that for $n\geq m \geq N$, $a_{m}-a_{n}<\epsilon$ and $a_{n}-a_{m}<\epsilon$. Thus for $n\geq m \geq N$, $-\epsilon < a_{n}-a_{m}<\epsilon$, so $|a_{m}-a_{n}|<\epsilon$.
\end{proof}
\begin{thm}
If a sequence $\{a_n\}^{\infty}_{n=1}$ in $\mathbb{R}$ is increasing or non-decreasing, then it is upward half Cauchy.
\end{thm}
\begin{proof}
Suppose a sequence $\{a_n\}^{\infty}_{n=1}$ in $\mathbb{R}$ is increasing or non-decreasing then for any $n\geq m$ with $n,m\in\mathbb{N}$, $a_{m}-a_{n}\leq 0$. Thus for every $\epsilon > 0$ take $N=1$, then for $n\geq m \geq N$, $a_{m}-a_{n}\leq 0 <\epsilon$. 
\end{proof}
\begin{thm}
If a sequence $\{a_n\}^{\infty}_{n=1}$ in $\mathbb{R}$ is decreasing or non-increasing, then it is downward half Cauchy.
\end{thm}
\begin{proof}
Suppose a sequence $\{a_n\}^{\infty}_{n=1}$ in $\mathbb{R}$ is decreasing or non-increasing then for any $n\geq m$ with $n,m\in\mathbb{N}$, $a_{n}-a_{m}\leq 0$. Thus for every $\epsilon > 0$ take $N=1$, then for $n\geq m \geq N$, $a_{n}-a_{m}\leq 0 <\epsilon$. 
\end{proof}
We have already shown that both monotone sequences and Cauchy sequences are half Cauchy. Now we will show that bounded half Cauchy sequences of real numbers converge.
The following lemma will shorten the proof of the forthcoming Theorem 4. The proof of this lemma is left as an exercise for the reader.
\begin{lemma}
A sequence $\{a_n\}^{\infty}_{n=1}$ in $\mathbb{R}$ is upward half Cauchy if and only if the sequence $\{-a_n\}^{\infty}_{n=1}$ in $\mathbb{R}$ is downward half Cauchy.
\end{lemma}
\begin{thm}
 A sequence $\{a_n\}^{\infty}_{n=1}$ in $\mathbb{R}$ converges if and only if $\{a_n\}^{\infty}_{n=1}$ is bounded and half Cauchy.
\end{thm}
\begin{proof}
Suppose that a sequence $\{a_n\}^{\infty}_{n=1}$ in $\mathbb{R}$ converges, then a standard analysis course tells us that $\{a_n\}^{\infty}_{n=1}$ is bounded and Cauchy. 
Theorem 1 tells us that this means $\{a_n\}^{\infty}_{n=1}$ is bounded and half Cauchy. Suppose that a sequence $\{a_n\}^{\infty}_{n=1}$ in $\mathbb{R}$ is bounded and downward half Cauchy.
Put $I=\liminf_{n\ra\infty}a_{n}$. Now, since $\{a_n\}^{\infty}_{n=1}$ is downward half Cauchy, for every $\epsilon > 0$ there exists $N\in \mathbb{N}$ so that $a_{n}-a_{m}<\frac{\epsilon}{2}$ for all $n\geq m\geq N$.
Moreover, there exists $M\geq N$ so that $a_{M}-I < \frac{\epsilon}{2}$. Otherwise $I+\frac{\epsilon}{2}$ is a lower bound for $\{a_n\}^{\infty}_{n=N}$, contradicting our definition of $I$.
So for $n\geq M$, $a_{n}\leq a_{M}+\frac{\epsilon}{2}\leq I+\epsilon$. So, to recap, for every $\epsilon > 0$ there exists $M\in \mathbb{N}$ so that for $n\geq M$, $a_{n}\leq I+\epsilon$.
Thus, $\limsup_{n\ra\infty}a_{n} \leq I$. So, $\{a_n\}^{\infty}_{n=1}$ converges. Alternatively, suppose that a sequence $\{a_n\}^{\infty}_{n=1}$ in $\mathbb{R}$ is bounded and upward half Cauchy. Then by Lemma 1, $\{-a_n\}^{\infty}_{n=1}$ in $\mathbb{R}$ is bounded and downward half Cauchy, hence convergent by the previous argument.
Since  $\{-a_n\}^{\infty}_{n=1}$ is convergent, $\{a_n\}^{\infty}_{n=1}$ is convergent.
\end{proof}

\section{An Example}
Consider the following system of recursive inequalities
$$0\leq x_{n}\leq x_{n-1} + y_{n-1},\quad n\in\mathbb{N},$$
$$0\leq y_{n}\leq \lambda y_{n-1},\quad n\in\mathbb{N}.$$
With $x_{0}$ and $y_{0}$ both positive and $0\leq \lambda < 1$. We will use the concept of half Cauchy sequences to prove that both $\{x_n\}^{\infty}_{n=1}$ and $\{y_n\}^{\infty}_{n=1}$ converge.
\begin{proof}
First, we prove by induction that $0\leq y_{n}\leq \lambda^{n}y_{0}$ for all $n\in\mathbb{N}$. The base case $n=0$ is clearly true. Now suppose $0\leq y_{N-1}\leq \lambda^{N-1}y_{0}$, then $0\leq y_{N}\leq \lambda y_{N-1} \leq \lambda(\lambda^{N-1}y_{0})= \lambda^{N}y_{0}$.
Since $0\leq \lambda < 1$, $y_{n}\ra 0$. We now prove by induction that for $k\leq n$, $0\leq x_{n}\leq x_{n-k} + \sum^{k}_{i=1}\lambda^{n-i}y_{0}$, the base case $k=1$ follows from our recursive inequality and the inequality $0\leq y_{n}\leq \lambda^{n}y_{0}$ for all $n\in\mathbb{N}$. Suppose now that for some $k< n$, $0\leq x_{n}\leq x_{n-k} + \sum^{k}_{i=1}\lambda^{n-i}y_{0}$, then by our recursive inequality
$$0\leq x_{n}\leq x_{n-k} + \sum^{k}_{i=1}\lambda^{n-i}y_{0}\leq x_{n-k-1} + y_{n-k-1} + \sum^{k}_{i=1}\lambda^{n-i}y_{0}\leq  x_{n-k-1} + \sum^{k+1}_{i=1}\lambda^{n-i}y_{0}.$$
Where the last inequality above again uses the inequality $0\leq y_{n}\leq \lambda^{n}y_{0}$ for all $n\in\mathbb{N}$. So we have shown that for $k\leq n$, $0\leq x_{n}\leq x_{n-k} + \sum^{k}_{i=1}\lambda^{n-i}y_{0}$. 
Letting $k=n$ we see $0\leq x_{n}\leq x_{0} + \sum^{n}_{i=1}\lambda^{n-i}y_{0}\leq x_{0} + \frac{y_{0}}{1-\lambda}$. So the sequence $\{x_n\}^{\infty}_{n=1}$ is bounded. Furthermore, for $n> m\geq N$ 
$$x_{n}-x_{m}\leq x_{n-(n-m)} + \sum^{n-m}_{i=1}\lambda^{n-i}y_{0}-x_{m} = \sum^{n-m}_{i=1}\lambda^{n-i}y_{0} = \lambda^{N}\left(\sum^{n-m}_{i=1}\lambda^{n-i-N}\right)y_{0} \leq \frac{\lambda^{N}y_{0}}{1-\lambda}.$$
Thus given $\epsilon>0$ there exists $N\in\mathbb{N}$ so that for $n\geq m\geq N$, $x_{n}-x_{m}<\epsilon$. So the sequence $\{x_n\}^{\infty}_{n=1}$ is bounded and downward half Cauchy, hence convergent by Theorem 4.
\end{proof}

\end{document}